\documentclass[10pt]{amsart}
\usepackage{amssymb,amsmath,graphicx,verbatim,eufrak,amsthm}
\usepackage{color}
\usepackage{xcolor}
\usepackage{wasysym}
\usepackage{hyperref}
\hypersetup{
    colorlinks=true,       
    linkcolor=violet,          
    citecolor=green,        
    filecolor=magenta,      
    urlcolor=cyan           
}

\newtheorem{thm}{Theorem}
\newtheorem{cor}[thm]{Corollary}

\newtheorem{conj}[thm]{Conjecture}

\newtheorem*{clm*}{Claim}

\theoremstyle{definition}

\theoremstyle{remark}
\newtheorem{rem}[thm]{Remark}

\newcommand{\be}[1]{\begin{equation}\label{#1}}
\newcommand{\ee}{\end{equation}}
\numberwithin{equation}{section}

\newcommand{\ba}[1]{\begin{align}\label{#1}}
\newcommand{\ea}{\end{align}}
\numberwithin{equation}{section}

\newcommand{\ben}{\begin{equation*}}
\newcommand{\een}{\end{equation*}}
\numberwithin{equation}{section}

\newcommand{\calP}{\mathcal{P}}

\newcommand{\bbN}{\mathbb{N}}

\newcommand{\bbP}{\mathbb{P}}

\newcommand{\bbR}{\mathbb{R}}

\newcommand{\bbZ}{\mathbb{Z}}

\newcommand{\rk}[1]{\bgroup\color{red}%
  \par\medskip\hrule\smallskip%
  \noindent\textbf{#1}%
  \par\smallskip\hrule\medskip\egroup}

\newcommand{\IP}[1]{\langle #1 \rangle}
\newcommand{\conn}{\leftrightarrow}
\newcommand{\Z}{\bbZ}
\newcommand{\R}{\bbR}
\newcommand{\N}{\bbN}

\newcommand{\ie}{{\em i.e.}\ }
\newcommand{\eg}{{\em e.g.}\ }

\newcommand{\arXiv}[1]{arXiv preprint: #1}

\newcommand{\1}[1]{\mathbf{1}_{\{ #1 \} } }

\newcommand{\p}{\partial}

\renewcommand{\Pr}{\mathbb{P}}

\title[Indicable groups and $p_c<1$]{Indicable groups and $p_c<1$}

\author[Aran Raoufi]{Aran Raoufi}
\author[Ariel Yadin]{Ariel Yadin}

\begin{document}
\maketitle

\begin{abstract}
A conjecture of Benjamini \& Schramm from 1996
states that any finitely generated group that is not a finite extension of $\Z$ 
has a non-trivial percolation phase.
Our main results prove this conjecture for certain groups, and in particular 
prove that any group with a non-trivial homomorphism into the additive group of real numbers 
satisfies the conjecture.
We use this to reduce the conjecture to the case of hereditary just-infinite groups.

The novelty here is mainly in the methods used, combining the methods of EIT and evolving sets, and 
using the algebraic properties of the group to apply these methods.
\end{abstract}

\section{Introduction}

Bernoulli percolation on a graph is the process where each edge of the graph is deleted or kept 
independently.  This model has its origin in statistical physics \cite{perc}, 
but gives rise to interesting and beautiful mathematics
even in ``non-realistic'' geometries, such as Cayley graphs of abstract groups.
Especially interesting in these cases is the relation between the algebraic properties of the group and 
the behavior of the percolation process.  One example which we do not tackle in this paper is 
the relation between existence of a non-uniqueness infinite component phase and amenability 
(for this, see \eg \cite{LyonsPeres, Gabor} and references therein).
In this paper we are concerned with the property of the existence of a non-trivial percolation phase,
usually known as ``$p_c<1$''.

We now introduce our results rigorously.

\subsection{Percolation on groups}

Let $G$ be a finitely generated group. 
Let $S$ be a finite symmetric generating set for $G$.
Let $\Gamma = \Gamma(G,S)=(V(G,S),E(G,S))$ be a right Cayley graph for $G$. 
(That is, the graph whose vertices are elements of 
$G$ and edges are defined by $x \sim y$ if $x^{-1} y \in S$.)
Denote the unit element of $G$ by $1$. 
Let $\bbP_p$ the Bernoulli site percolation measure with parameter $p$.
(See \cite{BolRior, Grimmett, LyonsPeres} for background on percolation.)
Let 
$x \conn \infty$ denote the event that $x$ is in an infinite component.
Let $p_c(\Gamma)$ be the critical point for percolation on $\Gamma$, i.e.,
$$p_c(\Gamma)= \inf \{ p \in [0,1] \ : \  \bbP_p[ 1 \leftrightarrow \infty] > 0 \}.$$

Since the property $p_c(\Gamma) < 1$ is invariant to quasi-isometries (see \eg Theorem 7.14 in \cite{LyonsPeres}), 
it does not depend on the specific choice of Cayley graph $\Gamma$.  
Thus, we may write $p_c(G) < 1$ without ambiguity. 
(This is in contrast to the fact that the specific value of $p_c(\Gamma)$ depends vey much on the
specific choice of Cayley graph, see \eg \cite[Chapter 3.3]{Grimmett}.)

\begin{conj}[Benjamini \& Schramm \cite{BS96}]
\label{conj:main}
For any finitely generated group $G$,
$p_c(G) = 1$ if and only if $G$ has a finite index cyclic subgroup.
\end{conj}

It is well known (see \cite{LyonsPeres}) that if $G$ has polynomial growth then the above conjecture is valid, 
for example by using Gromov's theorem regarding
groups of polynomial growth \cite{Gromov}, and the structure of nilpotent groups.  
Our results below give an alternative proof of this, 
which does not require the full theory of nilpotent groups 
(although there are other methods 
to prove this fact in the literature, see \cite{CT16}).
The conjecture is also known to hold for groups of exponential growth, due to Lyons \cite{Lyons95}.
Other works proving $p_c<1$ in the Cayley and non-Cayley graph setting include 
\cite{APS, BB99, CT16, Teix}.

Here is our main theorem.

\begin{thm} \label{thm:main1}
Let $G$ be a finitely generated group.
If there exists a finitely generated normal subgroup $N \lhd G$ 
with $\vert N \vert = \infty$ and $[G:N] = \infty$, then $p_c(G)<1$. 
%
\end{thm}

\subsection{Virtual characters}

A group property $\calP$ is a family of groups closed under isomorphism.
Examples of group property include Abelian groups, nilpotent groups, exponential growth groups.
We say a group $G$ is {\bf virtually} $\calP$, if $G$ has a subgroup of finite index that is in $\calP$.

By a {\bf character} of a group $G$ we refer to a non-trivial 
homomorphism from $G$ to $(\R,+)$ (the additive group of real numbers).
By a {\bf virtual character} of $G$ we mean a character of a finite index subgroup of $G$.
(A group admitting a character is sometimes called {\bf indicable}. 
$G$ admits a virtual character if and only if it is virtually indicable).

The above theorem implies the following corollary.
\begin{cor} \label{cor:car}
If $G$ admits a virtual character, then $p_c(G)<1$ unless $G$ contains a finite index infinite cyclic subgroup.
\end{cor}
Furthermore, by using the same methods, in the case that the group has a virtual character, we can show that if the group is transient itself, for $p$ sufficiently close to 1 the infinite cluster is transient. 

To make this statement precise, 
let us define for a graph $\Gamma$, 
$$ p_t(\Gamma) := \inf \{ p \in [0,1] \ : \ 
\infty \textrm{ clusters are transient $\bbP_p-$a.s. } 
\} . $$
Using indistinguishability \cite{indist}, this quantity is well defined; \ie infinite clusters are a.s.\ either all transient or all recurrent.
Of course $p_t(\Gamma) \geq p_c(\Gamma)$. 
Since transience is a property which is stable under quasi-isometries (see \eg Theorem 2.17 in \cite{LyonsPeres}), 
it follows that if $p_t(\Gamma) < 1$ for some Cayley graph $\Gamma$ of a finitely generated group $G$ 
then $p_t(\Gamma') < 1$ for any Cayley graph $\Gamma'$ of $G$.
Thus, as with $p_c$, we may write $p_t(G) < 1$ without ambiguity.

\begin{thm}\label{thm:p_t}
If $G$ admits a virtual character then $p_t(G)<1$ unless $G$ is virtually $\Z$ or virtually $\Z^2$.
\end{thm}

\subsection{A reduction}

Note that all groups of polynomial growth admit virtual characters (see \eg \cite{Kleiner}; 
this is in fact the standard main step toward proving Gromov's Theorem for polynomial growth groups).
But there are many other groups that admit virtual characters.
Indeed, any group with infinite Abelianization.

The Grigorchuk group is an example of a torsion group of intermediate growth.
Being torsion, it cannot admit a virtual character.  
However, Theorem \ref{thm:main1} still applies to the Grigorchuk group, and many other groups of intermediate growth.
In fact, most examples of intermediate growth groups known are so called {\em branch groups} (see \cite{branch groups}), 
for which it is quite simple to prove $p_c<1$:  If G is a branch group,
for any positive integer $d$, there exists a group $H$, such that G contains $H^d$ as a finite index subgroup.
Thus branch groups have a 
Cayley graph containing $\N^d$ as a subgraph.  See below for some more details.

In fact we can use the above results to reduce Conjecture \ref{conj:main} to a specific family of groups.
Albeit, these groups are exactly those which there is a lack of examples, so they are poorly understood in a sense.
To state the reduction, we introduce some notation.

A group $G$ is {\bf just-infinite} if any non-trivial quotient of $G$ is finite; that is, any 
non-trivial normal subgroup of $G$ is of finite index.
A standard example of a just-infinite group is $\Z$.
However, this property is not hereditary; that is, one can have a just-infinite group that has 
a finite index subgroup that is not just-infinite.  
A {\bf hereditary just-infinite} group is a group for which every finite index subgroup is just-infinite.
An example of such group is an infinite simple group.
(Recently infinite finitely generated simple groups of intermediate growth 
have been shown to exist in \cite{N16}.)
It is known that the only 
elementary amenable just-infinite groups are $\Z$ or the infinite dihedral group.
Specifically, these have an infinite cyclic group of finite index.  
See \cite{Grig14} for the proof.

Our reduction of Conjecture \ref{conj:main} is:
\begin{thm}
\label{thm:reduction}
If Conjecture \ref{conj:main} holds for the class of 
(finitely generated, sub-exponential growth, non elementary amenable)
hereditary just-infinite groups then the conjecture holds for all finitely generated groups.
\end{thm}

\begin{proof}
Assume first that $G$ is just-infinite.
Then  either $G$ is:
\begin{itemize}
\item Case (I): either a branch group, 
\item Case (II): or contains a subgroup of finite index that is the direct product $H^d = H \times \cdots \times H$
of $d \geq 1$ copies of a hereditary just-infinite group $H$. 
\end{itemize}
See \cite{branch groups} for background, definitions and the classification mentioned.

In Case (I), for any $d \in \mathbb{N}$, there exists a group $L$, such that $G$ contains a finite index subgroup of the form $L^d$ , see \cite{branch groups}.
Thus, in any branch group, for any $d$, the group admits a Cayley graph that contains a copy of $\N^d$.
Specifically, $p_c(G) < 1$ when $G$ is a branch group (in fact $p_t(G) < 1$).

In Case (II), if $d > 1$ then $G$ again has a Cayley graph that contains a copy $\N^2$, so $p_c(G) < 1$.

Thus, we are only left with the case where $G$ contains a finite index subgroup that is hereditary just infinite.
That is, we have shown that a just-infinite group $G$ with $p_c(G) = 1$ admits a finite index hereditary just-infinite group.

Now, if $G$ is a finitely generated infinite group, then there exists $N \lhd G$ such that 
$G/N$ is just-infinite.  See \eg Claim 2 in the beginning of Section 5 of \cite{B06} for 
a simple method of proving this.
If $p_c(G/N) < 1$ then $p_c(G) < 1$, by \cite[Theorem 1]{BS96}.
So assume that $p_c(G/N) = 1$.  Since $G/N$ is just-infinite, by the above it is hereditary just-infinite.
If Conjecture \ref{conj:main} holds for hereditary just-infinite groups, then $G/N$ has a finite index subgroup
isomorphic to $\Z$.  Thus, $G$ admits a virtual character, and Corollary \ref{cor:car} is applicable.
\end{proof}

{\bf Acknowledgement.} AR is supported by the NCCR SwissMAP, the ERC AG COMPASP, and the Swiss NSF.
AY is supported by the Israel Science Foundation (grant no.\ 1346/15).
This research was initiated while AY was visiting the Section of Mathematics, University of Geneva,
and the Centre Interfacultaire Bernoulli EPFL special semester ``Analytic and Geometric Aspects of Probability on Graphs''.  AY expresses gratitude to both the University of Geneva and the CIB for their wonderful hospitality and support.

\section{Probabilistic tools}

\subsection{EIT}

The proofs of our results are based on the method called \emph{EIT}, or 
{\em exponential intersection tails}. 
Let $\mu$ be a probability measure on the set of infinite paths 
on a graph, starting at some fixed origin.
We say that $\mu$ satisfies \emph{EIT}, if it has the following property:
\begin{itemize}
\item[(EIT)] There exists a constant $c>0$ such that for two independent paths $\gamma$ and $\gamma'$ with law $\mu$, 
and any $k\ge 1$,
$$ \mu \otimes \mu \big( \vert \gamma \cap \gamma' \vert \geq k \big)  \leq \exp(-ck). $$
\end{itemize}

This method was introduced in \cite{EIT}.
There it is shown that:
\begin{thm} \label{thm:EIT}
If there exists a measure $\mu$ satisfying \emph{EIT} on a graph $\Gamma$, then $p_c(\Gamma)<1$. Furthermore, $p_t(G)<1$.
\end{thm}

\subsection{Method of evolving sets}
The following is a consequence of Theorem 1.2 of 
Dembo, Huang, Morris, Peres \cite{DHMP16}.
Their theorem is proved using the method of {\em evolving sets} introduced by Morris and Peres in \cite{MP05}.
This is a method to bound the heat kernel decay via the isoperimetric properties of a graph.
See \cite{MP05} and \eg \cite[Chapter 8]{Gabor} for more details.
The theorem is basically stating the following rather intuitive fact:
If instead of walking according to some fixed time-independent transition matrix, 
one chooses some pre-determined time-dependent transition matrices, 
as long as these have some sort of ``uniform isoperimetric dimension'' at least $d$, 
then the heat kernel of this time-dependent walk must decay at most like that in $\Z^d$
(\ie of order at most $t^{-d/2}$).
In order to keep the notation as simple as possible, we do not state the theorem in its full generality, but rather 
tailored to the specific case we require it.


\begin{thm}
\label{thm:evolvingset}
Let $\Gamma_t = (V_t,E_t)$ be a sequence of connected graphs 
on a common vertex set $V_t=V$.  
We assume that the graphs $\Gamma_t$ are all isomorphic.

Denote the degree of $x \in V$ in the graphs $\Gamma_t$ by $\deg_t(x)$.
Suppose that $\sup_{t,x} \deg_t(x) < \infty$ (the degrees are uniformly bounded).
Suppose further that $\deg_t(x) = \deg_{t+1}(x)$ for all $t,x$ (the degrees of a vertex $x$ 
are constant in $t$).

Suppose further that all $\Gamma_t$ admit a $d$-dimensional isoperimetric inequality; 
that is, there exists $d>1$ such that for all $t$ and all non-empty finite sets $A \subset V$
we have
$$ | \p_t A |^d \geq |A|^{d-1} , $$
where 
$$ \p_t A = \{ \{x,y\} \in E_t \ : \ x \in A \ , \ y \not\in A \}  $$
is the edge boundary of $A$ in the graph $\Gamma_t$.

Fix some $\gamma > 0$ and 
consider the time-dependent Markov chain $(X_t)_t$ which has transition probabilities 
\begin{equation} \label{eq:RWpt} 
\Pr [ X_{t+1} = y \ | \ X_t = x ] =  \gamma \cdot \1{x=y} + (1-\gamma) \cdot  \tfrac1{\deg_t(x)} \cdot \1{ \{x,y\} \in E_t } .
\end{equation}

Then, there exist constants $K , C >0$ such that for all $s \geq 0 , t \geq 1$,
$$ \Pr [ X_{t+s} = y \ | \ X_s = x ] \leq C \cdot (K t)^{-d/2} . $$
\end{thm}

For the reader interested in checking the details of this reference, we provide a short 
``dictionary'' to translate Theorem 1.2 of \cite{DHMP16} into the above.
The $\gamma$ mentioned in Theorem \ref{thm:evolvingset} is the same $\gamma$ as in \cite{DHMP16}.
For every $t$, $\pi^{(t)}$ from \cite{DHMP16} is defined via $\pi^{(t)}(x,y) = \1{ \{x,y\} \in E_t }$.
Then, in (1.4) of \cite{DHMP16} we have $\beta(t) = 1$ for all $t$, because the degrees are constant in $t$.
Also, since all the graphs $\Gamma_t$ are isomorphic, we have that $\kappa_t$ from \cite{DHMP16}
is constant in $t$, and positive when $\Gamma_t$ admit a $d$-dimensional isoperimetric inequality.
Thus, for this $d$ we have that $\psi_{d,\beta}$ from \cite{DHMP16} admits
$\psi_{d,\beta}(t) = \kappa t$ for some $\kappa>0$.
(1.7) of \cite{DHMP16} then gives the assertion of Theorem \ref{thm:evolvingset}.

As a consequence of this theorem we have that:
\begin{cor}
\label{cor:evolving}
Under the conditions of Theorem \ref{thm:evolvingset}, let $(X_t)_t , (X'_t)_t$ be two independent copies
of the Markov chain defined in Theorem \ref{thm:evolvingset}. 

If for some $d>2$ the graphs $\Gamma_t$ admit a $d$-dimensional isoperimetric inequality,
then there exists a constant $c>0$ such that for all $k \geq 1$,
$$ \Pr \big[ | \{ t \ : \ X_t = X'_t \} | \geq k \big] \leq e^{-c k} . $$
\end{cor}

\begin{proof}
We will in fact prove the following:  For any fixed sequence $(x_t)_t$ we have for all $k \geq 1$,
$$ \Pr \big[ | \{ t \ : \ X_t = x_t \} | \geq k \big] \leq e^{-c k} . $$
This is essentially Lemma 3.1 from \cite{EIT}, and we include a short sketch only for completeness.

We choose $m$ to be large enough so that (by Theorem \ref{thm:evolvingset}) 
for any sequence $(x_t)_t$ and any $t$ we have 
$$ \sum_{t=1}^\infty \Pr [ X_{tm} = x_{tm} \ | \ X_0, \ldots, X_{t} ] \leq \sum_{t=1}^\infty C (K t (m-1))^{-d/2} =: \beta < 1 . $$
Thus, for any $t$,
$$ \Pr [ | \{ j \ : \ X_{t+jm} = x_{t+jm} \} | \geq r ] \leq \beta^r . $$
The conclusion follows readily.
\end{proof}

\begin{rem}
\label{rem:conjugate Cayley graphs}
A specific case where the conditions of Theorem \ref{thm:evolvingset} hold is the following:
Let $G$ be a finitely generated group, and let $N \lhd G$ be a finitely generated normal subgroup.
Suppose that for some (and hence every!) Cayley graph of $N$, we have a $d$-dimensional isoperimetric 
inequality.  Let $S$ be the symmetric generating set of $N$ inducing this Cayley graph.
For $x \in G$ note that $S^x = \{ x^{-1} s x \ : \ s \in S \}$ is again a generating set of $N$ (because $N$ is normal in $G$).
Also, the Cayley graph with respect to $S^x$ is isomorphic to the original Cayley graph with respect to $S$.
So for any fixed sequence $x_1,x_2,\ldots,$ we have that the sequence of Cayley graphs on $N$ 
induced by the generating sets $S^{x_t}$ are all isomorphic.  Such a sequence will adhere to the conditions of
Theorem \ref{thm:evolvingset}.
\end{rem}

\section{Proof of Theorem \ref{thm:main1}}

We separate the proof into three cases and treat each case separately: 
\begin{itemize}
\item 
Case 1: $N$ is not virtually $\bbZ$ nor virtually $\bbZ^2$.
\item
Case 2: $N$ is virtually $\bbZ^2$.
\item
Case 3: $N$ is  virtually $\bbZ$.
\end{itemize}

\subsection{$N$ is not virtually $\bbZ$ nor virtually $\bbZ^2$}
Let $H= G/N = \IP{ Nh_1^{\pm1}, Nh_2^{\pm1}, \dots, Nh_k^{\pm1} }$, where 
$G = \IP{ h_1^{\pm 1} , \ldots, h_k^{\pm 1} }$. Let $N=\IP{ n_1^{\pm1}, n_2^{\pm1}, \dots, n_\ell^{\pm1} }$. 
Let $\Gamma$ be the Cayley graph of $G$ with respect to the generators 
$\{h_i^{\pm1}, n_j^{\pm1}h_i^{\pm1} \ | \ i=1,\ldots, k  \ , \ j=1,\ldots, \ell \}$.

We construct a measure $\mu$ on the set of self-avoiding paths of $\Gamma$ and prove it satisfies EIT.  
First, fix a one-sided infinite self-avoiding path starting from the origin in the Cayley graph of $H$
with respect to the generators $\{ Nh_1^{\pm1}, Nh_2^{\pm1}, \dots, Nh_k^{\pm1} \}$.
Let this path be $(N u_j)_{j \ge 1}$,
where $u_j=s_1s_2 \cdots s_j$, and for each $i$, 
$s_i \in \{h_1^{\pm1}, h_2^{\pm1}, \dots, h_k^{\pm1} \}$. 
We emphasize that the path is self-avoiding in $H = G/N$, meaning $Nu_i = Nu_j$ if and only if $i=j$.
(Such a path can be chosen because $H=G/N$ is an infinite connected graph when viewed as a Cayley graph
with respect to the generators $\{ Nh_1^{\pm1}, Nh_2^{\pm1}, \dots, Nh_k^{\pm1} \}$.)

Define the measure $\mu$ on the paths as follows: Let $(X_i)_{i\geq 1}$ be a sequence of independent random variables each with  uniform distribution on the set $\{1, n_1^{\pm1}, n_2^{\pm1}, \dots, n_\ell^{\pm1} \}$. 
Define $\gamma(j)=X_1 s_1 X_2 s_2\dots X_j s_j$. Note that because of the choice of the generators, $\gamma = (\gamma(1),\gamma(2), \ldots,)$ is indeed a path on $\Gamma$. Hence, the measure on $(X_i)_{i\geq 1}$ induces a measure on the set of self-avoiding paths on $\Gamma$. Call this measure $\mu$. 

Now we prove that this measure $\mu$ satisfies EIT. 
Let $\gamma,\gamma'$ be two independent paths of law $\mu$, and write $\gamma = (\gamma(1), \gamma(2), \ldots), 
\gamma' = (\gamma'(1) , \gamma'(2) , \ldots)$ and 
$\gamma(j) = X_1s_1X_2s_2\dots X_ks_j$, $\gamma' (j) = X'_1s_1X'_2s_2\dots X'_ks_j$.
First notice that if $\gamma(i) = \gamma' (j)$ then $i=j$. 
This is due to the fact that $\gamma(i) \in Nu_i$ and $\gamma'(j) \in Nu_j$, and $(Nu_i)_i$ is a self-avoiding path on $H$. 

Define $u_i(x) = u_i x u_i^{-1}$.
$$Y_i= X_1 u_1(X_2)  u_2(X_3)  \cdots u_{i-1} (X_i)  \quad \textrm{ and } \quad 
Y_i' =  X_1' u_1(X_2')  u_2(X_3')  \cdots u_{i-1} (X_i') . $$

Since $\gamma(i) = Y_i u_i$, and similarly $\gamma'(i) = Y'_i u_i$,
we have that $\gamma(i)=\gamma'(i)$ if and only if $Y_i = Y_i'$.
Specifically,
\begin{equation} \label{eq:meetsametime}
 \{ \vert \gamma \cap \gamma' \vert \geq k \} = \{ \vert \{ i: Y_i=Y'_i  \}\vert \geq k \}.
 \end{equation}

$(Y_i)_i$ can be viewed as a discrete time Markov chain on $N$ with transition probability as in \eqref{eq:RWpt}, where $E_t$ is the edge set of Cayley graph of $N$ with generators $\{ u_{t-1}(n_1^+), u_{t-1}(n_1^-), \dots, u_{t-1}(n_\ell^+), u_{t-1}(n_\ell^-) \}$, and $\gamma = 1 /(2\ell+1)$. 
Eq. \eqref{eq:meetsametime} implies that in order to prove that $\mu$ satisfies EIT, it is enough to show that
$$\mu\otimes\mu \big( \vert \{ i: Y_i=Y'_i  \}\vert \geq k \big) \leq \exp (-ck).$$
The above is a direct consequence of Theorem \ref{thm:evolvingset} (see also Corollary \ref{cor:evolving} and Remark \ref{rem:conjugate Cayley graphs}) once the assumption of isoperimetric inequality is satisfied. 

The following well known theorem guarantees the required isoperimetric inequality and the assertion follows 
(see also \cite[Chapter 5.3]{Gabor}).

\begin{thm}[Gromov, also Coulhon, Sallof-Coste \cite{CSC}]
\label{thm:coulhon-salof}
Let $\Gamma$ be a Cayley graph of a finitely generated group $G$.
Let $B_r = | B(x,r) |$ be the number of elements in the ball of radius $r$ (with respect to the graph metric).
Define $\rho(n) = \min \{ r  \ : \ B_r \geq n \}$.
Then, for any non-empty finite set $A \subset V(\Gamma)$,
$$    | \p A |  \geq \frac{ |A| }{ 2 \rho( 2 |A| ) }  . $$
\end{thm}

If the volume growth of a Cayley graph $\Gamma$ is super-quadratic (\ie $B_r \geq c r^{d}$ for some $d>2$)
then $\rho(n) \leq C n^{1/d}$ and $\Gamma$ satisfies a $d$-dimensional isoperimetric inequality.
Gromov's theorem on groups of polynomial growth \cite{Gromov} together with the structure of nilpotent groups
imply that if the Cayley graph of $N$ does not have super-quadratic growth, then $N$ must be virtually nilpotent
with at most quadratic growth, which implies that $N$ must be either virtually $\Z^2$ or virtually $\Z$.

This completes the first case where $N$ is not virtually $\Z$ nor virtually $\Z^2$.

\subsection{$N$ is virtually $\bbZ^2$} 

In this case there exists a Cayley graph $\Gamma$ of $G$ which has a Cayley graph of $N$ 
as a subgraph. Note that $p_c(N)<1$ as $p_c(\mathbb{Z}^2) <1$, and because $N$ is a 
subgraph of $\Gamma$, this implies $p_c(\Gamma) < 1$.

\subsection{$N$ is virtually $\bbZ$}

$N$ admits a finite index infinite cyclic subgroup, which without loss of generality we can assume to be characteristic in $N$,
and thus normal in $G$.
By considering this normal subgroup of $G$, we may assume without loss of generality that $N$ is isomorphic to $\Z$.
Let $N=\IP{ n }$. 


First, we claim that there exists a finite index subgroup $G_1$ of $G$, such that $G_1$ commutes with $N$.
Indeed, $G_1$ acts on $N \cong \Z$ by conjugation, so for any $x \in G$ we must have $x^{-1} n x \in \{ n, n^{-1} \}$.
Let $G_1$ be the kernel of this map which is of index at most $2$.

Since $G_1$ is finite index in $G$, we may assume without loss of generality that $G=G_1$; that is, 
$N$ is in the center of $G$ ({\em i.e.}\ elements of $N$ commute with all elements of $G$).

We claim that there exists a Cayley graph of $G$ with $\mathbb{Z} \times \mathbb{N}$ as a subgraph. 
Thus, $p_c(G) \leq p_c(\mathbb{Z} \times \mathbb{N}) <1$ would follow. 

Let $H= G/N = \IP{ Nh_1^{\pm1}, Nh_2^{\pm1}, \dots, Nh_k^{\pm1} }$, where 
$G = \IP{ h_1^{\pm 1} , \ldots, h_k^{\pm 1} }$.  
Let $\Gamma$ be the Cayley graph of $G$ with respect to the generators 
$\{h_i^{\pm1}, n, n^{-1} \}$.
Like the first case, let $(N u_j)_{j \ge 1}$ be a self-avoiding path starting from the origin ($u_1 = 1$) in the Cayley graph of $H$
with respect to the generators $\{ Nh_1^{\pm1}, Nh_2^{\pm1}, \dots, Nh_k^{\pm1} \}$.
So, $u_j=s_1s_2 \cdots s_j$, and for each $i$, 
$s_i \in \{h_1^{\pm1}, h_2^{\pm1}, \dots, h_k^{\pm1} \}$. 

We embed the graph $\mathbb{Z} \times \mathbb{N}$ into $\Gamma$, by the mapping $\phi:\mathbb{Z} \times \mathbb{N} \rightarrow \Gamma$ defined as $\phi(i,j) = n^i u_j$. First note that $\phi$ is injective. If $n^{i_1} u_{j_1} = n^{i_2} u_{j_2}$, then projecting modulo $N$ would give us $N u_{j_1} = N u_{j_2}$, and because the path $(N u_j)_{j \ge 1}$ is self-avoiding we get $j_1 = j_2$. It implies that $n^{i_1} = n^{i_2}$, and hence $i_1 = i_2$. The map $\phi$ also maps two neighboring vertices to neighboring vertices. Indeed, 
$(i,j)$ and $(i,j+1)$ are mapped to two neighbors, because
$$\phi(i,j+1) = n^{i} u_{j+1} = n^{i} u_{j} s_{j+1} = \phi(i,j) s_{j+1},$$
and $s_{j+1}$ is in the generating set defining the Cayley graph $\Gamma$. 
Also, $(i+1,j)$ and $(i,j)$ are mapped to neighboring vertices. As $G$ and $N$ commute,
$$ \phi(i+1,j) = n^{i+1} u_{j} = n^{i} n u_{j} = n^{i} u_{j} n =\phi(i,j) n,$$
and $n$ belongs to the generating set of the Cayley graph. This concludes the proof of the embedding of $\Z \times \bbN$
into $\Gamma$.

\section{Groups with a virtual character}

First we mention the following theorem by Rosset (see also \cite{Grig14} for an extension).
\begin{thm}[Rosset \cite{Rosset}] \label{thm:Rosset}
If $G$ has sub-exponential growth and $N \lhd G$ such that $G/N$ is solvable, then $N$ is finitely generated.
\end{thm}

(See also \cite{Grig14} for an extension to the elementary amenable case.)

\begin{proof} [Proof of Corollary \ref{cor:car}]
If $G$ has exponential growth, as mentioned earlier, then $p_c(G)<1$. 
So, we can assume that $G$ has sub-exponential growth. 
By passing to finite index, we may assume without loss of generality that $G$ admits a character.

Let $N$ be a normal subgroup such that $G/N =\mathbb{Z}$.
Then Theorem \ref{thm:Rosset} implies that $N$ is finitely generated. 

If $N$ is infinite then Theorem \ref{thm:main1} is applicable.

If $N$ is finite: then $G$ acts on the finite subgroup $N$ by conjugation.  Let 
$K = \{ x \in G \ : \ \forall \ n \in N \ , \ x^{-1} n x = n \}$.
Then $K$ is normal in $G$ and of finite index (because $G/K$ embeds into permutations on $|N|$ elements).
By replacing $G$ with $K$, we may then assume without loss of generality that any element in $G$ commutes with any
element in $N$; that is, $N$ is central in $G$.
Now, since $G/N = \Z$, there exists $a \in G$ such that $\IP{Na} \cong \Z$.
Let $M = \IP{a}$.  This is an infinite subgroup of $G$ and since $a$ commutes with $N$ it is also normal in $G$.
Since $G/M$ is finite, we get that $G$ is virtually $\Z$ in this case.
\end{proof}

We now prove Theorem \ref{thm:p_t}.

\begin{proof}[Proof of Theorem \ref{thm:p_t}]
As $p_t<1$ is invariant under quasi-isometries, without loss of generality, we can assume $G$ has a character, and there exists $N$ such that $G/N=\mathbb{Z}$. If $G$ has exponential growth, it is known that $p_t<1$
(Lyons \cite{Lyons95} constructs a subgraph of some Cayley graph of $G$, which is a tree of 
exponential growth, taking a random geodesic on this tree results in an EIT measure).
So we can assume that $G$ has subexponential growth.  
Rosset's Theorem (Theorem \ref{thm:Rosset}) 
implies that $N$ is finitely generated. 

As in the proof of Corollary \ref{cor:car}, if $N$ is finite then $G$ is virtually $\mathbb{Z}$.

Suppose $N$ is not virtually $\mathbb{Z}$ or virtually $\mathbb{Z}^2$. The proof of Theorem \ref{thm:main1} constructs a measure which satisfies EIT. Hence, Theorem \ref{thm:EIT} implies that $p_t(G)<1$.

Before continuing, we mention a classical fact: If $G/N= \mathbb{Z}$, then $G \cong N \rtimes \mathbb{Z} $. 
This is true because $\mathbb{Z}$ is a free group. 
Now, if $N$ is virtually Abelian, then by passing to a finite index subgroup of $G$
we may assume without loss of generality that $N$ is Abelian.
Thus, $G \cong N \rtimes \Z$ is solvable of subexponential growth, which then must be 
virtually nilpotent by the classical results of Milnor \cite{Milnor} and Wolf \cite{Wolf}.
The Bass-Guivarch formula (\cite{Bass, Guiv}) implies that $G$ has polynomial growth.
By an unpublished argument of Benjamini \& Schramm, 
see \eg Theorem 9 of \cite{wedges}, if $G$ has polynomial growth then
either $G$ is virtually $\mathbb{Z}$ or virtually $\mathbb{Z}^2$, or $p_t(G)<1$. This concludes the proof.
%
%
%
%
\end{proof}

%

\bigskip 

\noindent
\begin{minipage}{0.46\textwidth}
	\footnotesize\obeylines
	\textsc{University of Geneva}
	\textsc{E-mail:} \texttt{aran.raoufi@unige.ch}\smallskip
\end{minipage}\hfill
\begin{minipage}{0.46\textwidth}
	\footnotesize\obeylines
	\begin{flushright}
		\textsc{Ben-Gurion University of the Negev}
		\textsc{E-mail:} \texttt{yadina@bgu.ac.il}
	\end{flushright}
\end{minipage}

\end{document}